\documentclass[10pt]{amsart}

\usepackage{amsmath, amsfonts, amsthm, amssymb, graphicx, fullpage, enumerate, float, caption, subcaption}
\usepackage{scalerel}
\usepackage{boldline} 
\usepackage{makecell} 
\usepackage{longtable} 
\usepackage{array}  
 
\newtheorem{theorem}{Theorem}[section]        
\newtheorem{lemma}[theorem]{Lemma}
\newtheorem{corollary}[theorem]{Corollary}  
\newtheorem{conjecture}[theorem]{Conjecture}

\newtheorem*{main theorem}{Main Theorem}

\theoremstyle{remark}     
\newtheorem*{rem}{Remark}   
\theoremstyle{definition}  
\newtheorem{definition}[theorem]{Definition}   
      
\def\N{\mathbb{N}}

\def\R{\mathbb{R}}     
\def\Z{\mathbb{Z}}

\def\norm#1{\|#1\|}

\begin{document}
\title{Lattice Configurations Determining Few Distances} 
\author{Vajresh Balaji, \quad Olivia Edwards, \quad Anne Marie Loftin, \quad Solomon Mcharo, \\ Lo Phillips,  \quad Alex Rice, \quad Bineyam Tsegaye}
 
\begin{abstract} We begin by revisiting a paper of Erd\H{o}s and Fishburn, which posed the following question: given $k\in \N$, what is the maximum number of points in a plane that determine at most $k$ distinct distances, and can such optimal configurations be classified? We rigorously verify claims made in remarks in that paper, including the fact that the vertices of a regular polygon, with or without an additional point at the center, cannot form an optimal configuration for any $k\geq 7$. Further, we investigate configurations in both triangular and rectangular lattices studied by Erd\H{o}s and Fishburn. We collect a large amount of data related to these and other configurations, some of which correct errors in the original paper, and we use that data and additional analysis to provide explanations and make conjectures. 
 
\end{abstract}

\address{Department of Mathematics, Millsaps College, Jackson, MS 39210}
\email{balajv@millsaps.edu} 
\email{edwarof@millsaps.edu} 
\email{loftiam@millsaps.edu}
\email{mcharsk@millsaps.edu}   
 \email{philllg@millsaps.edu} 
\email{riceaj@millsaps.edu} 
\email{tsegabl@millsaps.edu}

\maketitle 
\setlength{\parskip}{5pt}   

\section{Introduction}

In 1996, Erd\H{o}s and Fishburn \cite{EF} posed the following question: given $k\in \N$, what is the maximum number $g(k)$ of points in a plane that determine at most $k$ distinct distances, and can such optimal configurations be classified? Here and throughout we say that a set $E\subseteq\R^2$ \textit{determines} a  distance $\lambda>0$ if there exist two points $P,Q\in E$ satisfying $\norm{P-Q}=\lambda$, where $\norm{\cdot}$ is the standard Euclidean norm. In other words, $E$ determines at most $k$ distances if the cardinality of $\left\{\norm{P-Q}: P,Q\in E, P\neq Q\right\}$ is at most $k$.  By convention and for the remainder of this paper, $0$ is not counted as a distance determined by a set of points. 

The question of Erd\H{o}s and Fishburn can be thought of as a precise, small picture, inverse formulation of the famous Erd\H{o}s distinct distance problem, introduced by Erd\H{o}s \cite{Erdos} 50 years earlier, which concerns the minimum number $f(n)$ of distinct distances determined by $n$ points in a plane. In that original paper, Erd\H{o}s proved via an elementary counting argument that $f(n)=\Omega(\sqrt{n})$, and he conjectured that the correct order of growth is $n/\sqrt{\log n}$, as attained by a square subset of the integer lattice. After decades of incremental progress, the question of the asymptotic behavior of $f(n)$ was effectively resolved in a celebrated result of Guth and Katz \cite{GuthKatz}, who established that $f(n)=\Omega(n/\log n)$. However, the problem of precisely determining $g(k)$ for many $k$, and classifying optimal configurations, which we refer to as the \textit{Erd\H{o}s-Fishburn problem}, is still very much open for business. To aid our exploration, we introduce the following definitions.

\begin{definition} We use \textit{configuration} to refer to finite subsets of $\R^2$ modulo \textit{similarity}, meaning equivalence via scaling and distance-preserving transformation. For $n\in \N$, we let $R_n$ denote the configuration given by the set of vertices of a regular $n$-gon, and we let $R_n^+$ denote $R_n$ with an additional point at the center of the unique circle containing the vertices. For $k\in \N$, we say that a configuration is $k$-\textit{optimal} if it determines at most $k$ distinct distances and contains $g(k)$ points. We say that a configuration is \textit{EF-optimal} if it is $k$-optimal for some $k\in \N$.
\end{definition}

We start by observing that $g(1)=3$, and the only $1$-optimal set is $R_3$. To see this, we fix $U,V\in \R^2$, which after translation and rotation we can assume are $U=(-1,0)$, $V=(1,0)$. We note that in order to add any additional points without determining an additional distance, those points must lie on the circles of radius $2$ centered at $U$ and $V$, respectively. These two circles intersect at two points, $Q=(\sqrt{3},0)$ and $R=(-\sqrt{3},0)$, and since these two points are more than distance $2$ apart, we can only add one of them while maintaining a single distance.

For $k=2$, we see that $g(2)\geq 5$ by considering $R_5$, the vertices of a regular pentagon, but showing equality is nontrivial. More generally, by rotational symmetry, we see that $R_n$ determines $\lfloor n/2\rfloor$ distinct distances, and hence $g(k)\geq 2k+1$ by considering $R_{2k+1}$. Starting with $k=3$, we start to see a new player enter the picture, as both $R_7$ and $R_6^+$ are $7$-point configurations determining $3$ distances. 

Further, $R_6^+$ is particularly compelling, as it is a hexagonal array of points that lie in a lattice forming equilateral triangles. Foreshadowing future discussion, we introduce the following lattices.

\begin{definition} We let $L_{\Delta}=\left\{\left(a+\frac{1}{2}b,\frac{\sqrt{3}}{2}b\right): a,b\in \Z\right\}$, which we refer to as the \textit{triangular lattice}, and we let $L_{\Box}=\left\{(a,b): a,b\in \Z\right\}$ denote the usual integer lattice, which we refer to as the \textit{rectangular lattice}.

\end{definition}

\noindent We now give a synopsis of resolved cases, established by Erd\H{o}s and Fishburn \cite{EF} for $1\leq k \leq 4$, by Shinahara \cite{Shin} for $k=5$, and by Wei \cite{Wei} for $k=6$. The problem remains open for $k\geq 7$.

\begin{theorem} The following results are known for the Erd\H{o}s-Fishburn problem:

\begin{enumerate}[(i)] \item $g(1)=3$, and the only $1$-optimal configuration is $R_3$.  \item $g(2)=5$, and the only $2$-optimal configuration is $R_5$.  \item $g(3)=7$, and the only $3$-optimal configurations are $R_7$ and $R_6^+$.  \item $g(4)=9$, and there are four $4$-optimal configurations: $R_9$, two subsets of $L_{\Delta}$, and one additional.  \item $g(5)=12$, and the only $5$-optimal configuration is a subset of $L_{\Delta}$.  \item $g(6)=13$, and there are at least three $6$-optimal configurations: $R_{13}$, $R^+_{12}$,  and a subset of $L_{\Delta}$.

\end{enumerate}
 
\end{theorem}

We see that $k=5$ is the first case in which $R_{2k+1}$ is not a $k$-optimal configuration, which is quickly followed by the case $k=6$ in which $R_{13}$ and $R_{12}^+$ are $6$-optimal. However, the aforementioned fact that $n$ points can be arranged within a square subset of $L_{\Box}$ in order to determine $O(n/\sqrt{\log n})$ distinct distances ensures that $g(k)=\Omega(k\sqrt{\log k})$. Therefore, there exists $k_0\in \N$ such that neither $R_{2k}^+$ nor $R_{2k+1}$ is $k$-optimal for all $k\geq k_0$ (note that $R_{2k}$ is never $k$-optimal because $R_{2k+1}$ has more points with the same number of distinct distances). Following some illuminating constructions, Erd\H{o}s and Fishburn \cite{EF} indicate in a remark that one can take $k_0=7$, but the remaining details are left unverified. We also observe that at least one subset of $L_{\Delta}$ is $k$-optimal for $3\leq k \leq 6$. The appeal of $L_{\Delta}$ for the purposes of minimizing distinct distances is intuitive, based on the fact that the lattice forms equilateral triangles. The following conjecture of Erd\H{o}s and Fishburn makes precise the belief that $L_{\Delta}$ is the correct place to look for EF-optimal configurations.

\begin{conjecture}[Conjecture 1, \cite{EF}] \label{efcon} There exists at least one $k$-optimal configuration in $L_{\Delta}$ for all $k\geq 3$, and all $k$-optimal configurations are represented by subsets of $L_{\Delta}$ for all $k\geq 7$. 

\end{conjecture}

A mechanism by which Erd\H{o}s and Fishburn gather further information and provide evidence for Conjecture \ref{efcon} is the presentation of data on the number of distances determined by particular subsets of $L_{\Delta}$ and $L_{\Box}$. Specifically, they focus on hexagonal arrays in $L_{\Delta}$ and square arrays in $L_{\Box}$, particularly on cases when these arrays have approximately the same number of points. They observe that, in these cases, the hexagonal arrays of $L_{\Delta}$ have about $26\%$ fewer distances than their square counterparts.

\section{Results and Outline}


In Section \ref{rig}, by filling in gaps and constructing new examples, we rigorously verify the following result mentioned in the introduction ($k_0=7$), which is included in a remark without proof in \cite{EF}.
\begin{theorem} \label{EFver} $g(k)=2k+1$ for $k\in \{1,2,3,4,6\}$, while $g(k)>2k+1$ otherwise.  

\end{theorem}

\noindent Since $R_n$ determines $\lfloor n/2 \rfloor$ distinct distances and $R_n^+$ determines $n/2$ distinct distances if $n$ is divisible by $6$ and $\lfloor n/2 \rfloor+1$ distinct distances otherwise, Theorem \ref{EFver} resolves the question of when the vertices of a regular polygon, with or without an additional point in the center, is an EF-optimal configuration. 

\begin{corollary}$R_n$ is an EF-optimal configuration if and only if $n\in \{3, 5, 7, 9, 13\}$, and $R_n^+$ is an EF-optimal configuration if and only if $n\in \{6,12\}$.

\end{corollary}

In an effort to verify and expand upon the aforementioned data on lattice configurations provided in \cite{EF}, we made the surprising discovery that the data tables contain numerous, albeit relatively small, errors in the number of distinct distances determined by said configurations. After repeatedly and rigorously checking our code, we carried out some calculations by hand, the most readily doable of which concerned a $5\times 5$ square configuration in $L_{\Box}$, in other words $\{0,1,2,3,4\}^2$. The distances determined by this configuration are $\sqrt{n}$ for $n=1,2,4,5,8,9,10,13,16,17,18,20,25,32$, for a total of $14$ distinct distances, while the data table in \cite{EF} indicates only $13$ distinct distances.

 Further, we sought additional insight as to the signifcance of the $26\%$ figure observed by Erd\H{o}s and Fishburn when comparing the number of distances in comparably sized configurations in $L_{\Delta}$ and $L_{\Box}.$ In particular, we considered density and number theoretic properties, and investigated whether the hexagonal and square arrays were the best choices to compare the two lattices.  

 The following provides an outline of Section \ref{data}: 
\begin{itemize} \item[(i)] We provide corrected and expanded data concerning the number of points and distances determined by hexagonal arrays in $L_{\Delta}$ and square arrays in $L_{\Box}$. \\

\item[(ii)] We provide a heuristic explanation, through density and number theoretic considerations, for why an optimal configuration in $L_{\Delta}$ should be about $27.6\%$ better than an optimal configuration in $L_{\Box}$ for the purposes of minimizing distinct distances. This is close to the $26\%$ observed in \cite{EF}, although that observation was influenced by the errors in the data. We see that our heuristic is in fact rigorous in the case of intersections of the respective lattices with large disks centered at the origin. \\

\item[(iii)]  We provide new data indicating that, even amongst subsets of $L_{\Delta}$ and $L_{\Box}$, respectively, the hexagonal and square arrays are not optimal with regard to minimizing the number of distinct distances for a fixed number of points. We see that these configurations are routinely ``beaten" by the aforementioned lattice disks. We provide a large amount of numerical data for all four types of configurations, focusing on instances where all four types contain approximately the same number of points. \\

\item[(iv)] Using considerations from items (ii) and (iii), as well as Conjecture \ref{efcon}, we make a detailed conjecture related to the original Erd\H{o}s distinct distance problem. 

\end{itemize}

\section{Proof of Theorem \ref{EFver}} \label{rig}

For each integer $s\geq 2$, let $H_s$ denote the hexagonal array in $L_{\Delta}$ with $s$ points on each side. For clarity, we note that Figure \ref{hexproof} below depicts $H_4$. For the following two lemmas, we take the convention that the leftmost vertex of $H_s$ lies at the origin. 

\begin{lemma}\label{trans} Every distance that occurs in $H_s$ occurs between the origin and a point of $H_s$ in the closed upper half-plane $(y\geq 0)$. 
\end{lemma}

\begin{proof} Fix $P,Q\in H_s$. 

\noindent \textbf{Case 1:} Suppose $P$ is one of the six vertices of $H_s$. If $P$ is not the origin, $H_s$ can be rotated by an integer multiple of $60^{\circ}$ to take $P$ to the origin. If the image of $Q$ under this rotation lies in the lower half-plane,  we can reflect $H_s$ about the $x$-axis to take $Q$ to the upper half-plane. We note that, due to its symmetry, $H_s$ is invariant under this transformation, call it $\phi$. Further, since $\phi$ is distance-preserving, we have $\norm{\phi(Q)}=\norm{P-Q}$. 

\noindent \textbf{Case 2:} Suppose $P$ lies on the boundary of $H_s$ but is not one of the vertices. Since there are $s$ points on each side of the boundary, $P$ is at most distance $\lfloor \frac{s-1}{2}\rfloor$ away from the nearest vertex. We call this distance $d_1$, and we call the distance from $P$ to the opposite vertex on the same side $e_1$, hence $d_1+e_1=s-1$. As for $Q$, it lies on some edge of points parallel to the side of the boundary containing $P$. Along this parallel edge, $Q$ has two distances to the boundary of $H_s$, one in each direction, so we let $d_2$ denote the distance in the same direction as $d_1$, and we let $e_2$ denote the distance in the same direction as $e_1$. We note that the parallel edge containing $Q$ is at least as long as the side-length on the boundary, so $d_2+e_2\geq s-1=d_1+e_1$. Therefore, either $d_2\geq d_1$ or $e_2\geq e_1$. 

\noindent If $d_2 \geq d_1$, we can translate $P$ by $d_1$ units to a vertex, with $Q$ remaining inside $H_s$. Similarly, if $e_2\geq e_1$, we can translate $P$ by $e_1$ units to the other vertex, with $Q$ remaining inside $H_s$. In either case, we have reduced to Case 1. The translation procedure is illustrated in Figure \ref{hexproof} below.

\begin{figure} [H] 
		\includegraphics[scale=0.5]{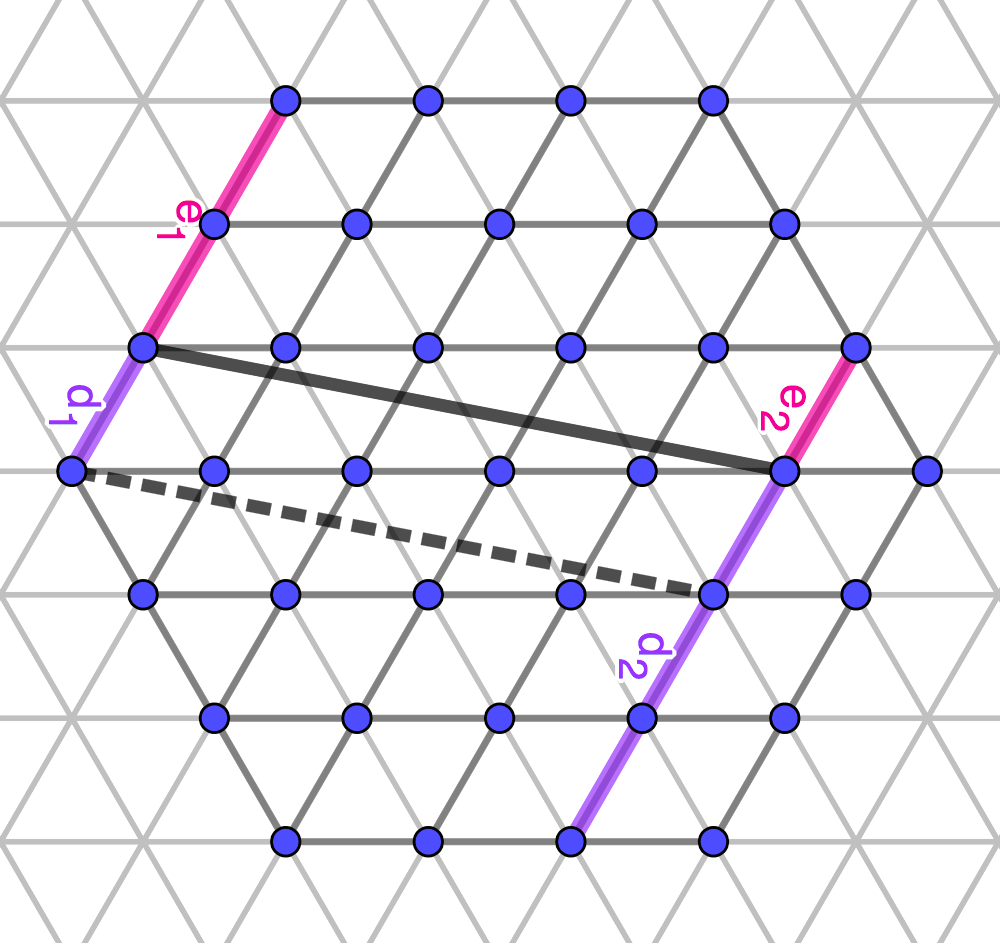} 
	 	\caption{The translation procedure described in Case 2 of the proof of Lemma \ref{trans}.}
		\label {hexproof} 		
\end{figure}

\noindent \textbf{Case 3:} If neither $P$ nor $Q$ lie on the boundary of $H_s$, then $P$ and $Q$ can both be translated left a unit at a time, preserving their distance, until one reaches the boundary, thus reducing to previous cases.
\end{proof}

\noindent As with the conclusion of Theorem \ref{EFver}, the following facts about $H_s$ were included in remarks in \cite{EF}, and we rigorously verify them here. 

\begin{lemma}\label{triup} $H_s$ contains $3s^2-3s+1$ points and determines at most $s^2-1$ distances.

\end{lemma}

\begin{proof}For the first part of the lemma, we see that we can decompose $H_s$ into a disjoint union of $\tilde{H}_j$, a boundary hexagon with $j$ points on each side, for $1\leq j \leq s$ (including $\tilde{H}_1$, which is a single point). By the inclusion-exclusion principle, the number of points in $\tilde{H}_j$ is $6j-6=6(j-1)$, with the exception of $|\tilde{H_1}|=1$. Therefore, $$|H_s|=1+\sum_{j=2}^s 6(j-1)=1+\sum_{j=1}^{s-1} 6j=1+6s(s-1)/2=3s^2-3s+1. $$ 

\noindent The cardinality $|H_s|$ is known as the $s$-th \textit{central hexagonal number}. For the second part of the lemma, we have from Lemma \ref{triup} that we only need to consider distances that occur from the origin to points in $H_s$ in the closed upper half-plane. Recall that points in $L_{\Delta}$ take the form $P=\langle a,b \rangle=a(1,0)+b(1/2,\sqrt{3}/2)$ for $a,b\in \Z$, in which case the distance from $P$ to the origin is $\sqrt{a^2+ab+b^2}$. We note that if the pair $\langle a,b \rangle$ occurs in $H_s$ with $b\geq a\geq 0$, then the pair $\langle b,a \rangle$ also occurs. Since the expression $a^2+ab+b^2$ is symmetric in $a$ and $b$, we can assume when counting distances from the origin to the upper half-plane in $H_s$ that $a\geq b$.
 
\noindent We see that in order to exhaust the upper half of $H_s$, we can first consider the pairs $\langle a,b \rangle$ for $0\leq a \leq s-1$ and $0\leq b \leq s-1$. What remains is a triangle of points on the right hand side, for which the allowable range of $b$ decreases as $a$ increases. Therefore, the total number of distances in $H_s$ is at most $$\sum_{j=2}^{s} j + \sum_{i=1}^{s-1} (s-i) = \sum_{j=2}^s j + \sum_{i=1}^{s-1} i = \frac{s(s+1)}{2}-1+\frac{(s-1)s}{2}=s^2-1.$$\end{proof}

\noindent In particular, for any $k\in \N$, we can let  $s=\lfloor \sqrt{k+1} \rfloor$, and Lemma \ref{triup} tells us that $H_s$ contains $3s^2-3s+1$ points and determines at most $s^2-1\leq k$ distances, hence $g(k) \geq 3s^2-3s+1$. We have therefore established the following corollary.

\begin{corollary}\label{floorcor} $g(k) \geq 3(\lfloor \sqrt{k+1} \rfloor)^2-3(\lfloor \sqrt{k+1} \rfloor)+1$ for all $k\in \N$.

\end{corollary}

\noindent Some basic algebra now reduces the proof of Theorem \ref{EFver} to a manageable number of remaining cases.

\

\begin{corollary}\label{gt63} $g(k)>2k+1$ for all $k\geq 63$.

\end{corollary}

\begin{proof} Fix $k \in \N$, and let $u=\sqrt{k+1}$. Since $\lfloor u \rfloor >u-1$, we have by Corollary \ref{floorcor} that \begin{equation*} g(k)>3(u-1)^2-3(u-1)+1=3u^2-9u+7. \end{equation*} Further, we see that $3u^2-9u+7\geq 2u^2-1=2k+1$ provided $u^2-9u+8=(u-8)(u-1)\geq 0$, which holds for $u\geq 8$, or in other words $k\geq 63$. \end{proof}

\begin{proof}[Proof of Theorem \ref{EFver}] Based on numerical data provided in \cite{EF} and re-verified, $H_s$ for $s\in \{3,4,5,6,7,8\}$ yield $(k,n)$ pairs, where $k$ is the number of distinct distances and $n$ is the number of points, of $$(8,19), (15,37), (23,61), (34,91) ,(46, 127), (59, 169).$$ Further, the same paper displays arrays yielding pairs $(7,16),(9,21),(10,25),(11,27),(13,31).$ In particular, we know that $g(k)>2k+1$ for $7\leq k \leq 17$, $23\leq k \leq 29$, $34\leq k \leq 44$, and $46\leq k \leq 83$. Further, we have by Corollary \ref{gt63} that $g(k)>2k+1$ for all $k\geq 63$. This leaves only the following list of exceptions: $$k\in \{18,19,20,21,22,30,31,32,33,45\}.$$   

\noindent The following four configurations, the number of distinct distances in which were verified both by hand and by computer, account for the remaining exceptions, and the theorem follows. \end{proof}

\begin{figure}[H]
\centering
\begin{subfigure}{.5\textwidth} 
  \centering
  \includegraphics[width=.5\linewidth]{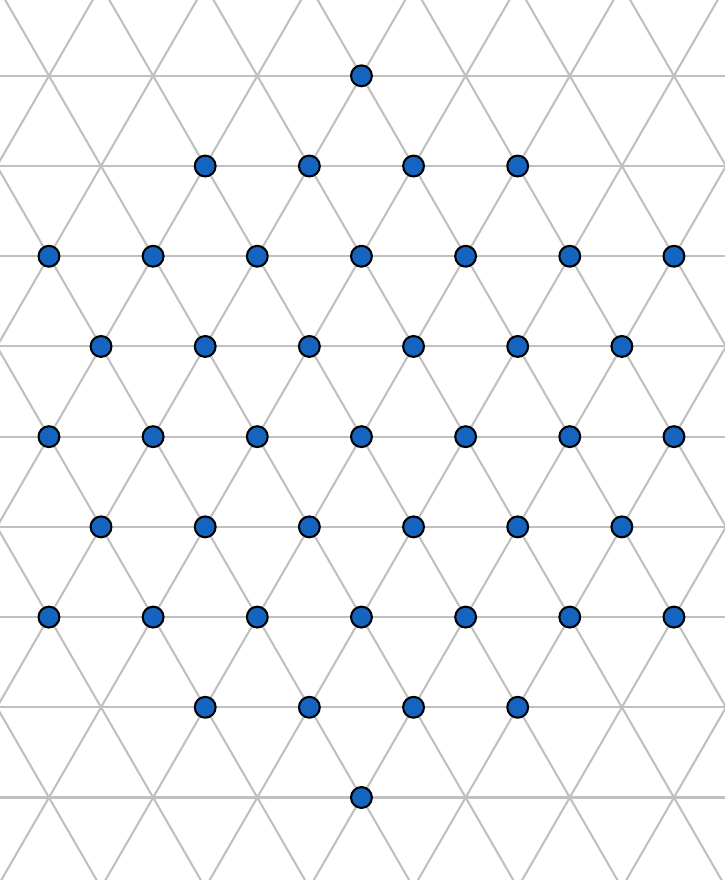}
  \caption{$k=18$, $n=43$}
  \label{18}
\end{subfigure}%
\begin{subfigure}{.5\textwidth}
  \centering
  \includegraphics[width=.5\linewidth]{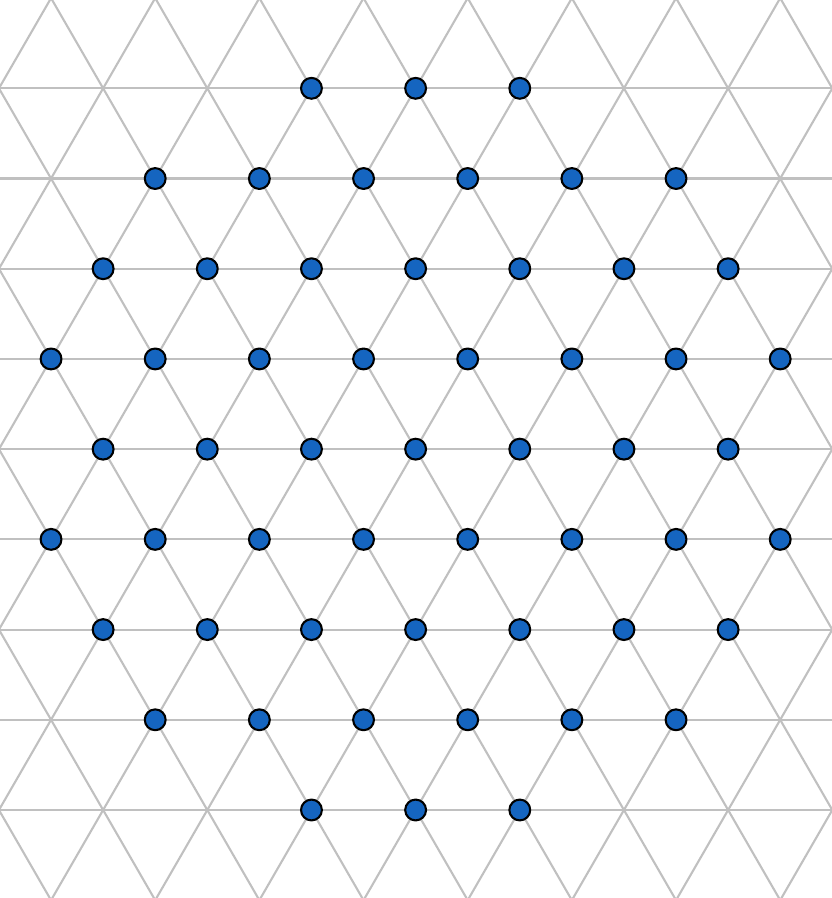}
  \caption{$k=21$, $n=55$}
  \label{22}
\end{subfigure}
\caption{These two configurations show that $g(k)>2k+1$ for $k\in \{18,19,20,21,22\}$.}\label{fig:test1}  
\end{figure}

\
 
  

\begin{figure}[H]
\centering
\begin{subfigure}{.5\textwidth}
  \centering
  \includegraphics[width=.6\linewidth]{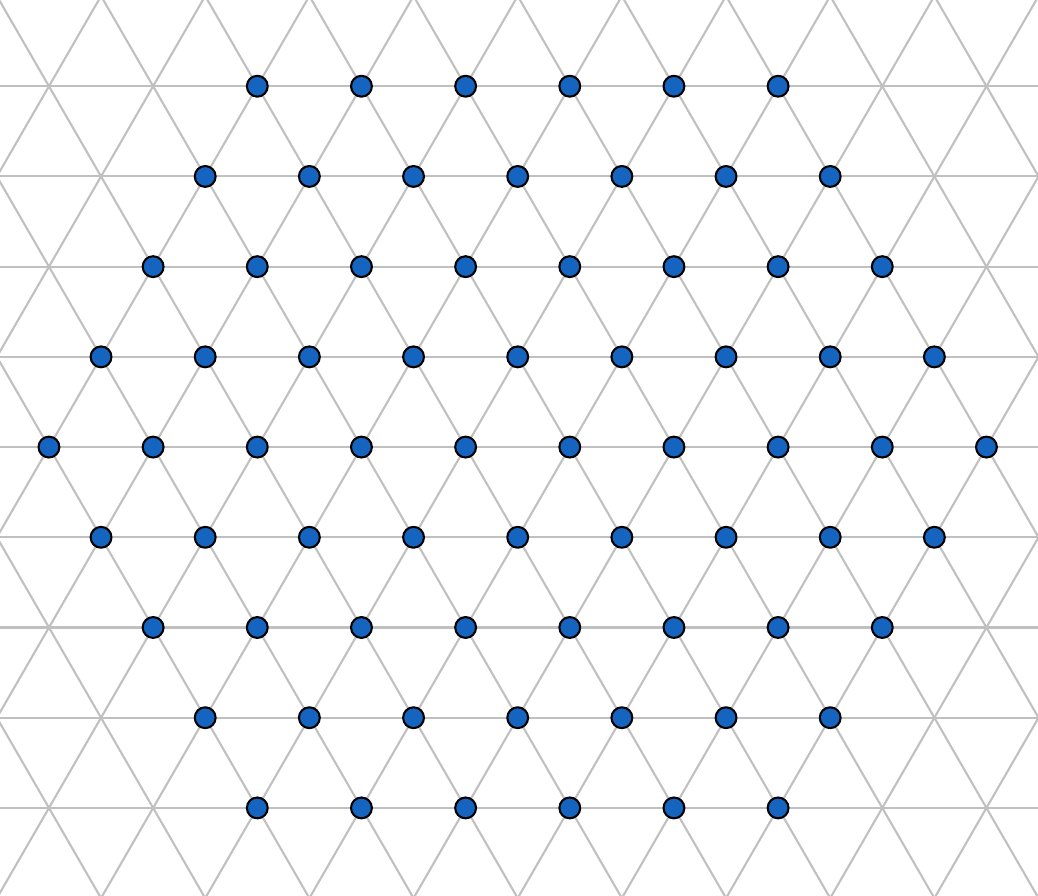}
  \caption{$k=29$, $n=70$}
  \label{29}
\end{subfigure}%
\begin{subfigure}{.5\textwidth}
  \centering
  \includegraphics[width=.6\linewidth]{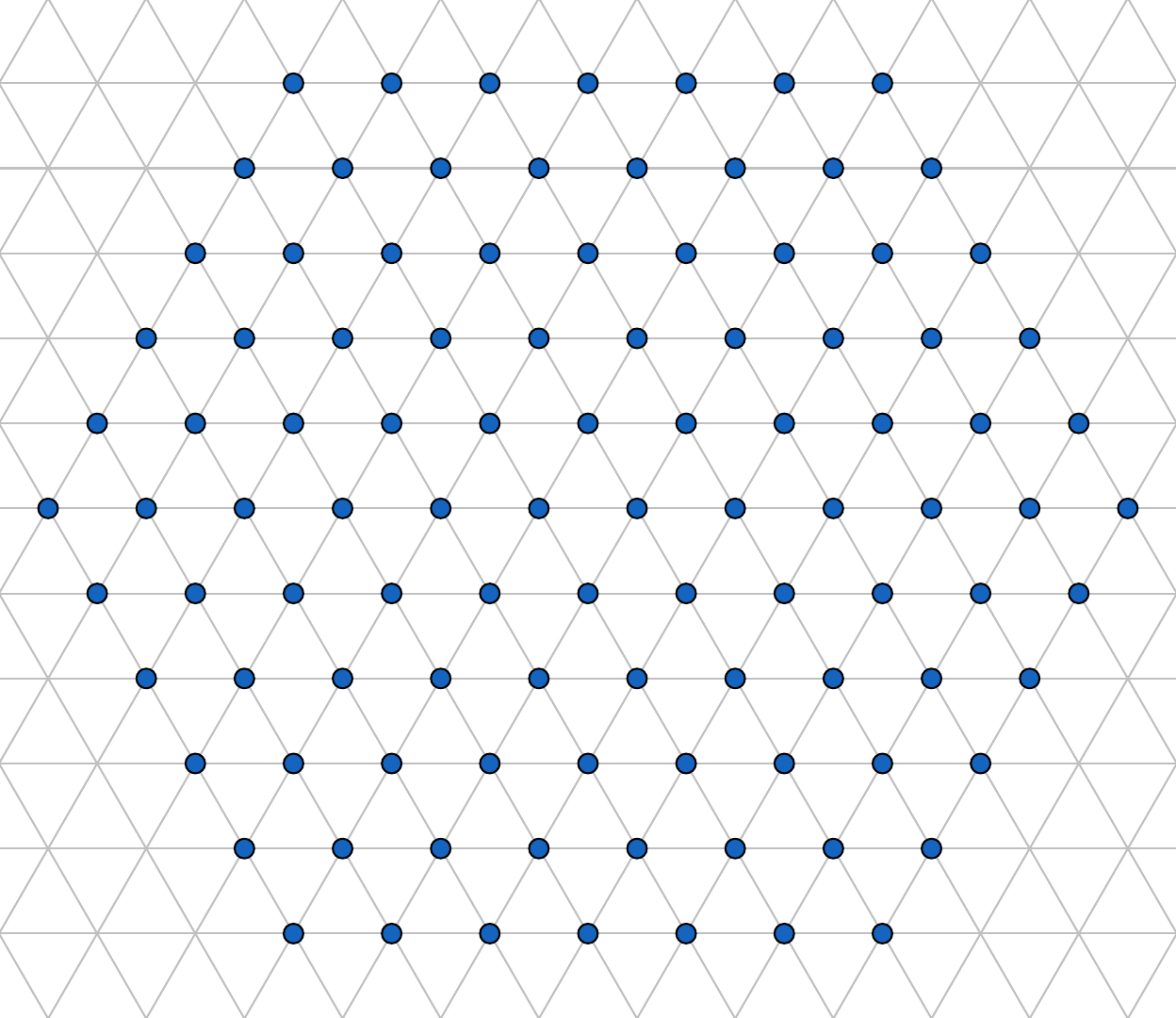}
  \caption{$k=40$, $n=102$}
  \label{40}
\end{subfigure}
\caption{These two configurations show that $g(k)>2k+1$ for $k\in \{30,31,32,33,45\}$.}\label{fig:test2}  
\end{figure} 
 


\newpage

\section{Numerical Data for Lattice Configurations} \label{data}

\subsection{Erd\H{o}s-Fishburn data revisited} We begin by presenting corrected and slightly expanded versions of the data tables from \cite{EF} containing the number of points and distinct distances determined by hexagonal arrays in $L_{\Delta}$ and square arrays in $L_{\Box}$.
\begin{center}
\begin{table}[H]
\caption{Number of points, $n$, and distinct distances, $k$, determined by hexagonal and square arrays in $L_{\Delta}$ and $L_{\Box}$, respectively. Corrected data is in bold, previously reported data from \cite{EF} is in parentheses. Data for $s=22,23$ for $L_\Delta$ and $36\leq s \leq 39$ for $L_{\Box}$ is new.}
\label{correct}
\begin{minipage}[c]{.48\linewidth}
\renewcommand{\arraystretch}{1.5}
\begin{tabular}{|| c | c | c ||| c | c | c ||}
\hline
\hline
\multicolumn{6}{||c||}{$H_s\subseteq L_\Delta$}\\ 
\hline
$n$ & $k$ & $s$ & $n$ & $k$ & $s$\\
\Xhline{0.8pt}
7&3&2&469&150&13\\
\hline
19&8&3&547&\textbf{172} (173)&14\\
\hline
37&15&4&631&\textbf{196} (197)&15\\
\hline
61&23&5&721&\textbf{222} (223)&16\\
\hline
91&34&6&817&\textbf{249} (250)&17\\
\hline 
127&46&7&919&\textbf{277} (280)&18\\
\hline
169&59&8&1027&\textbf{308} (312)&19\\
\hline
217&74&9&1141&\textbf{339} (345) &20\\
\hline
271&90&10&1261&\textbf{372} (382)&21\\
\hline
331&109&11&1387&405 &22\\
\hline
397&129&12&1519&440 &23\\
\hline

\hline
\end{tabular}
\end{minipage}
\begin{minipage}[c]{.48\linewidth}
\renewcommand{\arraystretch}{1.5}
\begin{tabular}{|| c | c | c ||| c | c | c ||}
\hline
\hline
\multicolumn{6}{||c||}{$s\times s$ square array in $L_\Box$}\\
\hline
$n$ & $k$ & $s$ & $n$ & $k$ & $s$\\
\Xhline{0.8pt}
4&2&2&441&\textbf{197} (194)&21\\
\hline
9&5&3&484&\textbf{215} (212)&22\\
\hline
16&9&4&529&\textbf{234} (228)&23\\
\hline
25&\textbf{14} (13)&5&576&\textbf{251} (248) &24\\
\hline
36&19&6&625&\textbf{272} (268)&25\\
\hline
49&\textbf{26} (25)&7&676&\textbf{293} (288)&26\\
\hline
64&\textbf{33}(32)&8&729&\textbf{314} (309)&27\\
\hline
81&\textbf{41}(40)&9&784&\textbf{336} (331)&28\\
\hline
100&\textbf{50}(49)&10&841&\textbf{359} (352)&29\\
\hline
121&\textbf{60} (58)&11&900&\textbf{381} (377) &30\\
\hline
144&\textbf{70} (69)&12&961&\textbf{407} (400)&31\\
\hline
169&\textbf{82} (80)&13&1024&\textbf{430} (425)&32\\
\hline
196&\textbf{93} (91)&14&1089&\textbf{456} (451)&33\\
\hline
225&\textbf{105} (104)&15&1156&\textbf{483} (474)&34\\
\hline
256&\textbf{119} (118)&16&1225&\textbf{507} (501)&35\\
\hline
289&\textbf{134} (130)&17&1296&535&36\\
\hline
324&\textbf{147} (146)&18&1369&566&37\\
\hline
361&\textbf{164} (160)&19&1444&594&38\\
\hline
400&\textbf{179} (177) &20&1521&623&39\\
\hline
\hline
\end{tabular}
\end{minipage}
\vspace{-4mm}
\end{table}
\end{center}

\noindent Data was collected using brute force searches with nested for-loops in Java. Specifically, we used that distances in the $s\times s$ square array in $L_{\Box}$ take the form $\sqrt{a^2+b^2}$ for $0\leq b\leq a \leq s-1$, while by Lemma \ref{trans}, distances in $H_s$ take the form $\sqrt{(a+\frac{1}{2}b)^2+(\frac{\sqrt{3}}{2}b)^2}=\sqrt{a^2+ab+b^2}$ for the pairs $\langle a,b \rangle$ indicated in the proof of Lemma \ref{triup}. In focusing on instances where the two arrays have approximately the same number of points, Erd\H{o}s and Fishburn indicate that $H_s$ determines about 26\% fewer distances. However, with the corrected table, using the respective $n$-values $(169,169), $  $(1027,1024)$, and the newly collected pair $(1519,1521)$, we find savings between 28\% and 29.4\% in the number of distinct distances determined by the hexagonal arrays in $L_{\Delta}$ versus the square arrays in $L_{\Box}$.

\subsection{A heuristic for $L_{\Delta}$ versus $L_{\Box}$} \label{heuristic} To gain a better understanding of the savings in distinct distances in $L_{\Delta}$ as compared to $L_{\Box}$, we start with two questions:

\begin{itemize} \item[(a)] How much ``less dense" than $L_{\Delta}$ is $L_{\Box}$? In other words, if a nice region  contains $n_1$ points of $L_{\Delta}$ and $n_2$ points of $L_{\Box}$, what would we expect $n_2/n_1$ to be?\\ \item[(b)] How much ``sparser" are the distances determined by $L_{\Delta}$ than the distances determined by $L_{\Box}$? In other words, if $k_1$ distances determined by $L_{\Delta}$ lie in some interval $(0,r]$, while $k_2$ distances determined by $L_{\Box}$ lie in $(0,r]$, what would we expect $k_1/k_2$ to be?

\end{itemize}

 Multiplying these two expectations together yields an expectation for $(n_2/k_2)/(n_1/k_1)$, which compares the ``efficiency" of $L_{\Box}$ to that of $L_{\Delta}$ with regard to maximizing the number of ``points per distinct distance". In particular, extending this predicted ratio beyond cases where the region determining $n_1$ and $n_2$ are the same, and in particular to cases where $n_1=n_2$, yields a prediction for $k_1/k_2$ in such cases.

 Of the two, Question (a) is the more straightforward, and is answered by considering the \textit{covolumes} of the lattices. Specifically, the number of points of $L_{\Box}$ lying inside of any nice region can be well-approximated by drawing a $1\times 1$ square centered at each point, yielding that $n_2$ is very close to the area of the region. For $L_{\Delta}$, we can instead draw parallelograms spanned by the vectors $(1,0)$ and $(1/2,\sqrt{3}/2)$ centered at each point, which have area $\sqrt{3}/2$, hence $n_1$ is approximately the area of the region divided by $\sqrt{3}/2$. Therefore, we predict that $n_2/n_1$ is approximately $\sqrt{3}/2$.

 Question (b) is in fact a question of number theory, specifically binary quadratic forms, as distances in $(0,r]$ determined by $L_{\Box}$ are in correspondence with integers $1\leq n \leq r^2$ that can be represented as $n=a^2+b^2$ for $a,b\in \Z$. Meanwhile, as discussed above, distances in $(0,r]$ determined by $L_{\Delta}$ are in correspondence with integers $1\leq n \leq r^2$ that can be represented as $n=a^2+ab+b^2$ for $a,b\in \Z$. 

\begin{rem} While our framing here is somewhat purpose-built for the Erd\H{o}s-Fishburn problem, the interested reader should note that these inquiries are closely related to a conjecture of Schmutz Schaller \cite{SS}, roughly stating that the triangular lattice determines the fewest distances of all lattices, resolved partially by Moree and te Riele \cite{Moree2} and later fully by Moree and Osburn \cite{Moree1}. \end{rem}

 It is known (see for example \cite{Moree1}) that the number of integers $1\leq n \leq N$ representable as $n=a^2+b^2$ is approximately $c N/\sqrt{\log N} $, where $$c=\left(\frac{1}{2}\cdot \prod_{p\equiv 3  (\text{mod }4)} \frac{p^2}{p^2-1} \right)^{1/2}\approx 0.764223654 $$ is known as the Landau-Ramanujan constant. Similarly (again see  \cite{Moree1}), the number of integers $1\leq n \leq N$ representable as $a^2+ab+b^2$ is approximately $c'N/\sqrt{\log N}$, where \begin{equation}\label{cprime} c'=\left(\frac{1}{2\sqrt{3}} \prod_{p\equiv 2  (\text{mod }3)} \frac{p^2}{p^2-1}\right)^{1/2}\approx 0.638909405,  \end{equation} and therefore we predict $k_1/k_2\approx c'/c \approx 0.83602. $

 Combining these considerations, we expect that if configurations in $L_{\Delta}$ and $L_{\Box}$ are optimal within their respective lattices, contain approximately the same number of points, and determine $k_1$ and $k_2$ distinct distances, respectively, then  $$\frac{k_1}{k_2}\approx \frac{\sqrt{3}}{2}\cdot\frac{c'}{c} \approx 0.72402. $$

 This hypothesized $27.6\%$ saving for $L_{\Delta}$ as compared to $L_{\Box}$ is close to and between the savings observed in \cite{EF} and the corrected data in Table \ref{correct}. However, neither the hexagonal arrays in $L_{\Delta}$ nor the square arrays in $L_{\Box}$ are known to be optimal in their respective lattices. The following section both rigorizes our heuristic in a special case, and introduces alternative candidates for optimal lattice configurations.

\subsection{Lattice disks} \label{disksec} The heuristic outlined in Section \ref{heuristic} makes some leaps. For example, the general geometric forms of optimal configurations in $L_{\Delta}$ and $L_{\Box}$ might be quite different, which would cast some doubt on the precision of the Question (a) analysis. Further, a lattice configuration does not have to determine \textit{every} distance determined in the full lattice lying in a particular interval, which makes the Question (b) analysis tenuous as well. 

However, there is a family of lattice configurations that completely alleviate these concerns, and also exploit rotational symmetry even more so than our previous candidate configurations: intersections of $L_{\Delta}$ and $L_{\Box}$ with large disks centered at the origin, which we refer to as \textit{lattice disks}. 

 With this in mind, we fix $n\in \N$. As discussed in Section \ref{heuristic}, the intersection of a disk of radius $r_1$ with $L_{\Delta}$ contains about $2\pi r_1^2/\sqrt{3}$ points. Setting this equal to $n$ yields $r_1=(\sqrt{3}n/(2\pi))^{1/2}$. Similarly, the intersection of a disk of radius $r_2$ with $L_{\Box}$ contains about $\pi r_2^2$ points, and setting this equal to $n$ yields $r_2=(n/\pi)^{1/2}$. A very slight overestimate for the number of distinct distances determined by the $L_{\Delta}$ disk is the number of integers $1\leq j \leq 4r_1^2=2\sqrt{3}n/\pi$ representable as $j=a^2+ab+b^2$. The missing exceptions stem from distances $\sqrt{a^2+ab+b^2}\leq 2r_1$ that cannot be translated to occur between two points of $L_{\Delta}$ within distance $r_1$ of the origin. Such distances are at least $2r_1-1$, so there are fewer than $4r_1<3\sqrt{n}$ of them.  

Therefore, the number of distances determined by the $L_{\Delta}$ disk is \begin{equation}\label{k1form} k_1= c'\frac{2\sqrt{3}n}{\pi\sqrt{\log n}}(1+o(1)), \end{equation} where $c'$ is as defined in Section \ref{heuristic} and the little-oh notation refers to $n$ tending to infinity. Similarly, the number of distances determined by the $L_{\Box}$ disk is $$k_2 = c\frac{4n}{\pi \sqrt{\log n}}(1+o(1)), $$ where $c$ is the Landau-Ramanujan constant, and both lattice disks contain $n+o(n)$ points. In this case we have, via rigorous argument rather than heuristic, that $$\frac{k_1}{k_2} = \frac{\sqrt{3}}{2} \cdot \frac{c'}{c} + o(1). $$

\subsection{Data for large lattice arrays and disks} The observations from the previous section are particularly interesting if we believe that lattice disks are good candidates for optimal or near-optimal sets with regard to the Erd\H{o}s-Fishburn problem or the original Erd\H{o}s distinct distance problem, or even if we believe them to be good candidates for optimal subsets of their respective lattices. To collect some evidence on this matter, we compare extensive data for the four families of lattice configurations on which we have focused: hexagonal arrays in $L_{\Delta}$, square arrays in $L_{\Box}$, $L_{\Delta}$ disks, and $L_{\Box}$ disks.

 To enhance our comparisons, we focus on the cases where these four arrays all have approximately the same numbers of points. As a starting point, we search via computer for values of $s$ such that $n_1=3s^2-3s+1$ is close to a perfect square, in the sense that $\sqrt{n_1}$ is within $.01$ of an integer $s_2$. Then, we let $n_2=s_2^2$, and the hexagonal array with $s$ points on a side and the square array with $s_2$ points on a side contain $n_1$ and $n_2$ points, respectively. Then, as calculated in the previous section, we let $r_1= (\sqrt{3}n_1/(2\pi))^{1/2}$ and we construct the $L_{\Delta}$ disk of radius $r_1$, which contains $n_3\approx n_2 \approx n_1$ points. Finally, we let $r_2=(n_1/\pi)^{1/2}$, construct the $L_{\Box}$ disk of radius $r_2$ containing $n_4\approx n_3 \approx n_2 \approx n_1$ points, and we compute the number of distances $k_1,k_2,k_3,k_4$, respectively, determined by each configuration. 

We compute $k_1$ and $k_2$ as outlined in Section \ref{data}, and we compute $k_3$ and $k_4$ via a brute force calculation of the number of integers $1\leq j \leq 4r_1^2$ representable as $a^2+ab+b^2$ and the number of integers $1\leq j \leq 4r_2^2$ representable as $a^2+b^2$, respectively. As discussed in Section \ref{disksec}, $k_3$ and $k_4$ are actually very slight overestimates for the number of distinct distances in the lattice disks, with relative error decaying quickly to $0$, and the error actually works in favor of our eventual conclusions and conjectures.  The following table displays the results of this data collection.
 
\begin{center}
\tiny 
\setlength{\extrarowheight}{2pt}
\begin{longtable}{||| p{1cm} | p{1cm} | p{1cm} || p{1cm} | p{1cm} | p{1cm} || p{1cm} | p{1cm} || p{1cm} | p{1cm} |||}
\caption{Number of points, $n_1,n_2,n_3,n_4$, and distinct distances, $k_1,k_2,k_3,k_4$, determined by hexagonal arrays in $L_{\Delta}$, square arrays in $L_{\Box}$, $L_{\Delta}$ intersected with a disk centered at the origin, and $L_{\Box}$ intersected with a disk centered at the origin, respectively.}\\
\hline
\multicolumn{3}{||| c ||}{$H_s\subseteq L_\Delta$}&\multicolumn{3}{c||}{$s\times s$ square in $L_\Box$}&\multicolumn{2}{c||}{$L_\Delta$ Disk }&\multicolumn{2}{c|||}{$L_\Box$ Disk}\\
\hline
$s$ & $n_1$ & $k_1$ & $s_2$ & $n_2$ & $k_2$ & $n_3$ & $k_3$ & $n_4$ & $k_4$\\
\hline
\hline
\endfirsthead
\hline
$s$ & $n_1$ & $k_1$ & $s_2$ & $n_2$ & $k_2$ & $n_3$ & $k_3$ & $n_4$ & $k_4$\\
\hline
\hline
\endhead
23&1519&440&39&1521&623&1519&441&1513&601\\
\hline
34&3367&925&58&3364&1310&3369&920&3360&1251\\
\hline
38&4219&1139&65&4225&1620&4217&1130&4216&1541\\
\hline
49&7057&1844&84&7056&2628&7059&1818&7049&2486\\
\hline
64&12097&3063&110&12100&4378&12094&3008&12083&4116\\
\hline
75&16651&4136&129&16641&5923&16634&4055&16641&5552\\
\hline
79&18487&4572&136&18496&6558&18482&4475&18480&6122\\
\hline
90&24031&5847&155&24025&8397&24036&5725&24010&7836\\
\hline
105&32761&7841&181&32761&11291&32755&7663&32759&10496\\
\hline
120&42841&10115&207&42849&14568&42848&9870&42841&13528\\ 
\hline
131&51091&11958&226&51076&17246&51097&11661&51096&15986\\
\hline
135&54271&12660&233&54289&18268&54263&12354&54248&16919\\
\hline
146&63511&14707&252&63504&21196&63519&14325&63509&19640\\
\hline
161&77281&17716&278&77284&25597&77289&17253&77268&23658\\
\hline
172&88237&20099&297&88209&29034&88230&19574&88223&26838\\
\hline
176&92401&21007&304&92416&30348&92406&20446&92332&28029\\
\hline
187&104347&23588&323&104329&34095&104352&22949&104340&31468\\
\hline
191&108871&24557&330&108900&35517&108864&23888&108869&32759\\
\hline
202&121807&27333&349&121801&39539&121812&26580&121785&36454\\
\hline
217&140617&31345&375&140625&45371&140619&30474&140616&41797\\
\hline
228&155269&34463&394&155236&49901&155273&33482&155260&45930\\
\hline
232&160777&35627&401&160801&51610&160771&34614&160760&47489\\
\hline
243&176419&38926&420&176400&56379&176421&37815&176391&51880\\
\hline
247&182287&40162&427&182329&58217&182317&39015&182265&53527\\
\hline
258&198919&43663&446&198916&63291&198916&42405&198912&58161\\
\hline
273&222769&48642&472&222784&70564&222768&47234&222761&64804\\
\hline
284&241117&52465&491&241081&76114&241114&50935&241093&69888\\
\hline
288&247969&53901&498&248004&78211&247946&52316&247959&71786\\
\hline
299&267307&57926&517&267289&84051&267323&56199&267302&77117\\
\hline
314&294847&63614&543&294849&92358&294851&61715&294821&84691\\
\hline
329&323737&69586&569&323761&101045&323735&67473&323676&92604\\
\hline
340&345781&74110&588&345744&107620&345756&71853&345771&98630\\
\hline
344&353977&75796&595&354025&110084&353981&73485&353961&100859\\
\hline
355&377011&80509&614&376996&116943&376986&78044&376976&107118\\
\hline
370&409591&87167&640&409600&126667&409562&84483&409575&115964\\
\hline
381&434341&92219&659&434281&133986&434343&89356&434308&122650\\
\hline
385&443521&94093&666&443556&136760&443552&91166&443497&125135\\
\hline
396&469261&99288&685&469225&144351&469249&96221&469208&132091\\
\hline
400&478801&101248&692&478864&147198&478792&98093&478776&134661\\
\hline
411&505531&106665&711&505521&155110&505541&103325&505521&141834\\
\hline
426&543151&114254&737&543169&166216&543137&110682&543138&151934\\
\hline
437&571597&119982&756&571536&174558&571633&116212&571602&159538\\
\hline
441&582121&122120&763&582169&177705&582125&118258&582072&162357\\
\hline
452&611557&128035&782&611524&186289&611562&123979&611530&170223\\
\hline
467&652867&136335&808&652864&198432&652878&131989&652825&181216\\
\hline
482&695527&144892&834&695556&210885&695520&140239&695508&192565\\
\hline
493&727669&151302&853&727609&220271&727659&146454&727647&201095\\
\hline
497&739537&153654&860&739600&223722&739555&148736&739542&204224\\
\hline
508&772669&160264&879&772641&233368&772601&155110&772639&212981\\
\hline
523&819019&169514&905&819025&246864&819023&164028&818990&225234\\
\hline
534&853867&176436&924&853776&256919&853874&170719&853843&234419\\
\hline
538&866719&178995&931&866761&260727&866686&173179&866699&237799\\
\hline
549&902557&186081&950&902500&271056&902576&180050&902467&247218\\
\hline
553&915769&188706&957&915849&274937&915768&182569&915747&250711\\
\hline
564&952597&195999&976&952576&285616&952579&189601&952567&260389\\
\hline
579&1003987&206178&1002&1004004&300439&1004013&199414&1003960&273859\\
\hline
590&1042531&213780&1021&1042441&311537&1042535&206744&1042495&283924\\
\hline
594&1056727&216578&1028&1056784&315648&1056726&209449&1056677&287637\\
\hline
605&1096261&224343&1047&1096209&327044&1096288&216960&1096225&297965\\
\hline
609&1110817&227205&1054&1110916&331216&1110835&219713&1110783&301764\\
\hline
620&1151341&235163&1073&1151329&342863&1151324&227412&1151304&312340\\
\hline
635&1207771&246273&1099&1207801&359046&1207777&238116&1207741&327015\\
\hline
646&1250011&254539&1118&1249924&371136&1249983&246085&1250012&337989\\
\hline
650&1265551&257572&1125&1265625&375609&1265572&249021&1265528&342035\\
\hline
661&1308781&266045&1144&1308736&387968&1308792&257188&1308692&353251\\
\hline
676&1368901&277784&1170&1368900&405139&1368882&268534&1368881&368843\\
\hline
691&1430371&289825&1196&1430416&422777&1430355&280107&1430362&384729\\
\hline
702&1476307&298749&1215&1476225&435779&1476291&288738&1476282&396591\\
\hline
706&1493191&302062&1222&1493284&440674&1493181&291932&1493127&400964\\
\hline
717&1540117&311203&1241&1540081&454022&1540131&300711&1540106&413066\\
\hline
732&1605277&323826&1267&1605289&472567&1605270&312938&1605233&429846\\
\hline
743&1653919&333283&1286&1653796&486344&1653950&322046&1653912&442387\\
\hline
747&1671787&336742&1293&1671849&491469&1671808&325387&1671718&446972\\
\hline
758&1721419&346375&1312&1721344&505556&1721389&334664&1721383&459723\\
\hline
762&1739647&349927&1319&1739761&510746&1739653&338071&1739625&464399\\
\hline
773&1790269&359715&1338&1790244&525054&1790299&347523&1790221&477381\\
\hline
788&1860469&373283&1364&1860496&544913&1860505&360604&1860478&495385\\
\hline
799&1912807&383420&1383&1912689&559698&1912800&370371&1912756&508782\\
\hline
803&1932019&387140&1390&1932100&565161&1932048&373949&1932003&513712\\
\hline
814&1985347&397407&1409&1985281&580262&1985338&383870&1985309&527326\\
\hline
818&2004919&401149&1416&2005056&585848&2004941&387510&2004882&532335\\
\hline
829&2059237&411672&1435&2059225&601097&2059246&397582&2059183&546210\\
\hline
844&2134477&426140&1461&2134521&622344&2134456&411547&2134454&565400\\
\hline
855&2190511&436926&1480&2190400&638008&2190527&421951&2190443&579686\\
\hline
859&2211067&440867&1487&2211169&643919&2211058&425754&2211025&584894\\
\hline
870&2268091&451826&1506&2268036&659933&2268113&436313&2268056&599420\\
\hline
885&2347021&466988&1532&2347024&682105&2347029&450917&2347002&619478\\
\hline
900&2427301&482326&1558&2427364&704669&2427293&465744&2427243&639850\\
\hline
911&2487031&493838&1577&2486929&721402&2487049&476764&2487014&655022\\
\hline
915&2508931&497986&1584&2509056&727521&2508908&480809&2508918&660584\\
\hline
926&2569651&509611&1603&2569609&744588&2569702&491991&2569595&675978\\
\hline
941&2653621&525623&1629&2653641&768155&2653614&507466&2653571&697226\\
\hline
952&2716057&537570&1648&2715904&785494&2716052&518955&2716009&713014\\
\hline
956&2738941&541948&1655&2739025&792020&2738981&523156&2738918&718819\\
\hline
967&2802367&554074&1674&2802276&809668&2802353&534812&2802306&734811\\
\hline
971&2825611&558517&1681&2825761&816285&2825597&539085&2825575&740685\\
\hline
982&2890027&570769&1700&2890000&834110&2890025&550893&2890023&756949\\
\hline
997&2979037&587712&1726&2979076&858983&2979000&567254&2979027&779382\\
\hline

\end{longtable}

\vspace{-7mm}

\end{center}

\noindent As our configurations grow, the originally investigated ratio $k_1/k_2$ decreases well below our previous observations and indicates at least a $31.5\%$ saving for hexagonal arrays in $L_{\Delta}$ versus square arrays in $L_{\Box}$. We know that the ratio $k_3/k_4$ must converge to $\frac{\sqrt{3}}{2}\cdot \frac{c'}{c}\approx 0.724$, and our largest data point yields a ratio of about $0.728$. The speed of the convergence is somewhat at the mercy of the convergence of the approximations for the density of the images of the respective binary quadratic forms, which is known to be quite slow. 

 Importantly, we see that for large configurations, $k_3$ is notably smaller than $k_1$ (about a $3.5\%$ saving, and climbing), and $k_4$ is notably smaller than $k_2$ (over a $9\%$ saving, and climbing). The fact that switching from a square to a disk had a bigger impact in $L_{\Box}$ than that of switching from a hexagon to a disk in $L_{\Delta}$ can probably be attributed to a greater increase in rotational symmetry. This discrepancy explains why $k_1/k_2$ falls well below $k_3/k_4$, and well below our heuristic ratio.

 Based on this data, and our intuition regarding the advantages of rotational symmetry over well-structured arrays, we conjecture that, with regard to minimizing distinct distances for a fixed number of points, or equivalently maximizing the number of points for a fixed number of distinct distances, the $L_{\Delta}$ and $L_{\Box}$ disks, or efficient subsets thereof, are the optimal configurations within their respective lattices. Combining this belief with Conjecture \ref{efcon} and the formulas (\ref{cprime}) and (\ref{k1form}), we conclude our discussion with the following detailed conjecture on the original Erd\H{o}s distinct distance problem. 
 
 \

\begin{conjecture} If $f(n)$ is the minimum number of distinct distances determined by $n$ points in a plane, then $$f(n) =c\frac{n}{\sqrt{\log n}}(1+o(1)),$$ where $$c=\frac{1}{\pi}\left(2\sqrt{3} \prod_{p \equiv 2 \textnormal{ mod }3} \frac{p^2}{p^2-1} \right)^{1/2}\approx 0.704498\dots $$
 
\end{conjecture}
 



\noindent \textbf{Acknowledgements:} This research was initiated during the Summer 2019 Kinnaird Institute Research Experience at Millsaps College. All authors were supported during the summer by the Kinnaird Endowment, gifted to the Millsaps College Department of Mathematics. At the time of submission, all authors except Alex Rice were Millsaps College undergraduate students. The authors would like to thank Pieter Moree and an anonymous referee for their helpful comments and references.

\end{document}